\documentclass{article}
\usepackage{mypkg}
\usepackage{utfmaths}
\newcommand{\Johnson}{\mathcal{J}}

\newcommand{\GenPhrase}{jumbled Johnson graph}

\title{Model theory of Johnson Graphs}
\begin{document}
\maketitle
\begin{abstract}

A \emph{Johnson graph} $J(n,k)$ is a graph whose vertices are the $k$-element subsets of an underlying set of size $n$ where two vertices are adjacent if they intersect in a subset of size $k-1$.
We extend the definition to allow infinite parameters and
we study a natural limit theory $T_{\Johnson}$ of Johnson graphs. We show that
$T_{\Johnson}$ is complete and we give an axiomatisation of $T_{\Johnson}$, from which
we obtain a complete description of the models of $T_{\Johnson}$.

We use this to show tameness properties of $T_{\Johnson}$, namely that $T_{\Johnson}$ is $\omega$-stable of Morley rank $\omega$, but not monadically dependent.
In particular it follows that the class of all finite Johnson graphs is stable, and therefore dependent but not monadically dependent.


\end{abstract}

\section{Introduction}
Johnson graphs are an easily-described $2$-parameter family of graphs with very strong symmetry and regularity properties.
The goal in this paper is to describe, from the viewpoint of model theory, the first order logic of these graphs as \emph{both} parameters grow.

A \emph{Johnson graph},  $J(n,k)$ is a graph whose vertices can be identified with the $k$-element subsets of an set of size $n$ with two vertices adjacent if they intersect in a subset of size $k-1$. We refer to the set of size $n$ as the \emph{underlying set} of the Johnson graph.
It is immediate that the symmetric group $\operatorname{Sym}(n)$ on the underlying set acts as
a group of automorphisms of the Johnson graph, and in fact if $n>2k$ then $\operatorname{Sym}(n)$  is the full automorphism group \cite{AutomorphismGroupOfJohnsonGraph}.
As a result, Johnson graphs are the most easily described non-trivial examples of \emph{distance-transitive} graphs, that is, graphs whose automorphism group is transitive on the set of ordered pairs
of vertices at any given distance.
Johnson graphs have been extensively explored from the viewpoint of graph theory and permutation group theory, but little from that of model theory.
In \cite{AdlerBM22} the VC-dimension and the VC-density of the edge relation on Johnson graphs are studied.
It is there shown that on the class of all finite Johnson graphs the edge relation has VC-dimension $4$ and VC-density $2$.

It is easily seen that for fixed $k$, the Johnson graph $J(n,k)$ is uniformly (as $n$ varies) bi-interpretable with a pure set of size $n$.
Thus $J(ω,k)$, a Johnson graph with a countably infinite underlying set and vertices of size $k$, has a totally categorical theory (of Morley rank $k$).
In this paper, we explore the limit theory as $n,k,n-k$ all grow.
One might expect the limit theory to behave more like the Boolean algebra of subsets of $ω$, which is far from stable.
However, we show that this limit theory is complete and $\omega$-stable of Morley rank $\omega$, and describe its models.
We see this as a first step towards exploring limit theories of other families of finite distance-transitive graphs.

In this paper we will use a natural extension of the definition of Johnson graphs that allows infinite parameters.
\begin{definition}
\label{InfJohnsonDef}
Let $κ$ and $λ$ be cardinals such that $ω ≤\lambda\leq κ$.
The \emph{Johnson graph} $J(κ,λ)$ is a graph whose vertices are all the size $λ$ subsets $A$  of an underlying set $X$ of size $κ$ with the complement $X∖A$ of size $κ$
such that
two vertices $u$ and $v$ are adjacent if and only if $|u∖v|=|v∖u|=1$.
\end{definition}

In the above definition we fix the size of the complements of vertices.
We do this to make sure that all of the vertices have complements of the same cardinality.
This is to ensure that the graphs have vertex transitive automorphism groups, in particular, that the symmetric group on the underlying set $X$ acts vertex transitively.
An equivalent definition is obtained by fixing the size of the vertices to be $κ$ and the complements of vertices to have size $λ$.
The choice between the two definitions is arbitrary, and our choice is analogous to using the natural isomorphism between  $J(n,k)$ and $J(n,n-k)$ to assume $n≥2k$.

For completeness we include here the usual definition of Johnson graphs, which covers finite parameters.
\begin{definition}
\label{JohnsonDef}
Let $n$ and $k$ be natural numbers such that $k≤n$.
The \emph{Johnson graph} $J(n,k)$ is a graph whose vertices are all the size $k$ subsets $A$  of an underlying set $X$ of size $n$
such that
two vertices $u$ and $v$ are adjacent if and only if $|u∖v|=|v∖u|=1$.
\end{definition}

Throughout this paper we identify vertices with subsets of the underlying set.

When $k>n$ the definition of $J(n,k)$ degenerates so for any fixed $n$ we get a finite number of Johnson graphs and no interesting limit theory.
So the main issue is when $n,k, n-k \to \infty$.
%
To gain a deeper understanding of the infinite Johnson graphs we will be investigating what their individual connected components look like.
\begin{definition}
	Let $κ$ and $λ$ be  cardinals such that $κ≥λ$.
	Take a set $A⊆X$ such that $|A|=λ$, $|X|=κ$ and $|A|≤|X∖A|$.
	We define a  graph  \emph{$J'(κ,λ)$} in the following way.
	There are vertices for each set of the form $(A∖B)∪C$ where
	$B⊂A$, $C⊂X∖A$ and $|B|=|C|∈ℕ$.
	Note that by choosing $B=C=∅$ we have a vertex for $A$ which we call the central vertex of $J'(κ,λ)$.
	Two sets $u,v$  are adjacent if and only if $|u\symdiff v|=2$.
\end{definition}
Note that for finite $κ$ and $λ$ we have that $J'(κ,λ)=J(κ,λ)$.

We will be working in the language $L$ of graphs, i.e. a language with a single binary relation symbol $E$ which we call the \emph{edge relation}.
We shall use an Ehrenfeucht-Fraïssé game argument to explore the  limit theory of the finite Johnson graphs $J(n,k)$ as $\min\{k,n-k\}\to ∞$, and we denote this theory by $T_{\Johnson}$, i.e.
         $T_{\Johnson}$ is the set of sentences $σ$ in $L$ such that for some $N_{σ}∈ℕ$, if $\min\{n,k,n-k\}≥ N_{σ}$, then $J(n,k)\models σ$.
\begin{theorem}
	\label{EF-Intro}
	Duplicator has a winning strategy in the following Ehrenfeucht-Fraissé games:
	\begin{itemize}
		\item  $EF_{\floor{\log{t}}}(J(n,k),J(n',k'))$ where $t=\floor{\frac{\min(k,n-k,k',n'-k')}{2}}$.
		\item  $EF_{\floor{\log{t}}}(J(n,k),\bigsqcup_{i∈I} J'(κ_i,λ_i))$ where $(κ_i)_{i∈I}$, $(λ_i)_{i∈I}$ are families of infinite cardinals with $ω≤λ_i≤κ_i$ and $t=\floor{\frac{\min(k,n-k)}{2}}$.
	\end{itemize}
\end{theorem}
This gives us the following result.
\begin{corollary}
\label{CompleteTheory}
$T_{\Johnson}$ is a complete theory.
\end{corollary}
Note that the infinite Johnson graphs of Definition 1.0.1 can be written as disjoint unions of grahps $J'(κ,λ)$ for fixed $λ$ and $κ$. This suggests the following definition.
\begin{definition}
	\label{GenJDef}
	Let $μ$ be some ordinal.
	A graph $G$ is a \emph{\GenPhrase} if
	$G=\bigsqcup_{i∈μ} J'(κ_i,λ_i)$ where  $(κ_i)_{i∈μ}$  and  $(λ_i)_{i∈μ}$ are sequences  of infinite cardinals such that $λ_i≤κ_i$ for all $i∈μ$.
\end{definition}

From Theorem \ref{EF-Intro} it follows that {\GenPhrase}s are models of $T_{\Johnson}$.
We then move on to axiomatizing $T_{\Johnson}$, and for simplicity we do this in two steps.
First we work in an extension $L_c$ of the language $L$ of graphs that also has three constant symbols and give a set $Σ_c$ of axioms in $L_c$ which hold of all connected models of $T_{\Johnson}$.
We then show that by treating those constant symbols like free variables and binding each with a universal quantifier we obtain a complete set $Σ$  of axioms in $L$ for $T_{\Johnson}$.
This gives us the following results.
\TODO[LOOK into labeling, this is Theorem LAxThm and corollaries GenJohnsonThm and TJAxThm]
\begin{theorem}\label{AxiomResultsIntro}
	\begin{enumerate}[before=\leavevmode\vspace{-\baselineskip}\vspace{-\topsep}]
	\item
		A connected graph $M$ has $M\models Σ $ if and only if $M\cong J'(κ,λ)$ for some infinite cardinals $κ≥λ$.
	\item
		A model $M$ satisfies $M\models Σ$ if and only if $M$ is a \GenPhrase.
	\item
		$Σ$ is a complete set of axioms for $T_{\Johnson}$.
	\end{enumerate}
\end{theorem}

Once we have this complete description of the models of $T_{\Johnson}$ we prove the following result.
\begin{theorem}
	The theory $T_{\Johnson}$ is $ω$-stable of Morley rank $ω$.
\end{theorem}

We complement this result by showing that $T_{\Johnson}$ is not monadically dependent (and hence not monadically stable). Recall that a theory $T$ is {\em monadically dependent} if any expansion of any model of $T$ by unary predicates is dependent. Similarly, a theory $T$ is {\em monadically stable} if any expansion of any model of $T$ by unary predicates is stable.

These notions also play a crucial role in model-checking for first-order logic on classes of graphs.  It is known that model checking is fixed-parameter tractable on every monadically stable graph class. Furthermore, monadic stability is precisely the dividing line between tractability and intractability of model-checking on hereditary classes that exclude a half-graph as a semi-induced subgraph, up to some reasonable complexity-theoretic assumptions~\cite{DreierEMMPT24}.
It is conjectured that a hereditary graph class $\mathcal C$ allows fixed-parameter tractable model-checking if and only if $\mathcal C$ is monadically dependent.
Under reasonable complexity-theoretic assumptions, the left-to-right direction has recently been proved for hereditary graph classes, with progress made in the right-to-left
direction~\cite{PrzybyszewskiTorunc-FlippingandForking}.
Our results indicate that Johnson graphs lie just outside the tractable world.

Section 2 introduces notation, definitions and presents a few lemmas we will rely on.
In Section 3 we use Ehrenfeucht-Fraïssé arguments to show that $T_{\Johnson}$ is a complete theory.
In Section 4 we will give axioms for $T_\Johnson$.
Section 5 contains the stability-theoretic characteristics of $T_{\Johnson}$.
We conclude in Section 6 with possible future research questions.

\section{Preliminaries}
In our previous paper \cite{AdlerBM22} we developed a technique of working with Johnson graphs by viewing the neighbourhoods of each vertex.
The neighbourhood of a single vertex belongs to a well known class of graphs called rook's graphs.
\begin{definition}
Let $m,n$ be cardinals.
The \emph{rook's graph} $R(m,n)$ is the cartesian product of $K_m$ and $K_n$.
We can represent every vertex by a pair from $R\times C$ where $|R|=m$ and $|C|=n$
and two distinct vertices $(i,j),(k,l)$ are adjacent if and only if $i=k$ or $j=l$.
\end{definition}

The name of these graphs comes from the intuition that if we take a $m$ by $n$ chessboard and place a vertex on each square then there is an edge between two squares if a rook can legally be moved from one to the other.
The following is straightforward to verify.
\begin{lemma}[\cite{AdlerBM22}]
	\label{nRook}
	In a Johnson graph $J(n,k)$ the open neighbourhood of any vertex $v$ induces $R(n-k,k)$ as a subgraph of $J(n,k)$.

\end{lemma}

\begin{lemma}
\label{EquivRelationLemma}
	Let $R$ be the relation on the powerset of some set $X$ such that $aRb$ if and only if $|a∖b|=|b∖a|∈\mathbb{N}$.
	Then $R$ is an equivalence relation.
\begin{proof}
For finite sets $R$ is equivalent to equinumerosity, and for $X$ infinite  this is a more restricted relation but we note that if we have $aRb$ and $bRc$ then $|a∖c|≤|a∖b|+|b∖c|$ and thus $|a∖c|∈ℕ$. For more details see Lemma 5.1.3 in \cite{Thesis}.
\end{proof}

\end{lemma}

\begin{lemma}
	Any vertex in  $J'(κ,λ)$ can act as the central vertex.
\begin{proof}
	Let $M=J'(κ,λ)$ with central vertex $A$ and underlying set $X$ and  $B$ be  any other vertex of $M$.
	Then let $M'$ be $J'(κ,λ)$ with central vertex $B$ and underlying set $X$.
	By Lemma \ref{EquivRelationLemma} we have that $M$ and $M'$ have the same vertex set namely the equivalence class of $R$  that includes  $A$ and $B$.
	The interpretation of the edge relation does not depend on the selection of central vertex in any way so it is identical in  $M$ and $M'$.
\end{proof}
\end{lemma}
In particular, the above lemma says that $J'(κ,λ)$ is vertex-transitive. In fact it is easily seen to be distance-transitive, though this is not used here.

Let $X$ be a set and $n$ an integer such that $n≤|X|$.
Then $\binom{X}{n}$ denotes set of all $n$-element subsets of $X$.
Let $v$ be a vertex in a graph $G$.
We call the set of all vertices adjacent to $v$  the \emph{neighbourhood of $v$}, denoted $N(v)$. 
Since we do not allow loops on vertices we have $v∉N(v)$, and
 what we call here neighbourhood is also sometimes called the \emph{open neighbourhood} around~$v$.
Let $v$ be a vertex in a graph $G$. We call the set of all vertices of distance at most $r$ from~$v$  the \emph{ball of radius $r$ around $v$}. Observe that $v$ is at distance $0$ from itself and thus included in all balls around~$v$.

We will now introduce Ehrenfeucht-Fraïssé games.
Ehrenfeucht-Fraïssé game arguments are used for determining if a concept is expressible in first order logic, or if two structures satisfy the same sentences.
The game is played between two players, usually named Spoiler and Duplicator.
\begin{definition}
	Let $L$ be a language with no function symbols, $n∈ℕ$, and  $M$ and $M'$ be $L$-structures.
	The \emph{Ehrenfeucht-Fraïssé game} $EF_n(M,M')$ is a $2$ player game played in the following way.

	In each round first Spoiler plays one element from either $M$ or $M'$, then Duplicator plays one element from whichever model Spoiler did not play in.
	We call the element played in $M$ on the $i$-th round $a_i$ and the element played in $M'$ we call $b_i$ regardless of which player plays them.
	After $n$ rounds have been played we have $(a_i)_{i=1}^n$ and $(b_i)_{i=1}^n$.

	Duplicator has won if for any atomic $L$-formula $ϕ$ and any sequence $(i_j)_{j=1}^k$ with $k≤n$ we have
	$M \models ϕ((a_{i_j})_{j=1}^k) \Leftrightarrow M'\models ϕ((b_{i_j})_{j=1}^k)$,
	otherwise Spoiler has won.

\end{definition}

Ehrenfeucht-Fraïssé games are a finite analogue to so-called back-and-forth arguments common in model theory.
Ehrenfeucht and Fraïssé showed the following.
	\begin{theorem}[Ehrenfeucht-Fraïssé {\cite{Ehrenfeucht1961}}, see also \cite{EF95}]
		\label{EF}
		If Duplicator has a winning strategy in $EF_n(M,M')$ then for all $L$-sentences $ϕ$ of quantifier depth at most $n$,
		$M\models ϕ \Leftrightarrow M'\models ϕ$.
	\end{theorem}

We note that if Duplicator has a winning strategy in $EF_n(M,M')$  for all $n∈ℕ$ then the models $M$ and $M'$ are indistinguishable from the perspective of first order logic.
This means that any characteristic that $M$ has and $M'$ does not have  cannot be expressed in first order logic.

%
%
%

\section{Complete infinite theory}
In this section we will explore the theory of infinite Johnson graphs, proving Theorem \ref{EF-Intro} and thereby obtaining Corollary \ref{CompleteTheory}.
We begin with a lemma.
\newcommand{\img}{\operatorname{img}}
\newcommand{\dom}{\operatorname{dom}}
\begin{lemma}
	\label{NoCover}
	The vertex set of $J(n,k)$  with $1<k≤\frac{n}{2}$ cannot be covered by $m-1$ balls of
	radius at most $t/(2^{m+1})$, for any $m\leq \lfloor \log t \rfloor$, where $t=k/2$.
	\begin{proof}
		The graph $J(n,k)$ has diameter $k$, and two vertices are distance $k$ apart whenever they are
		disjoint subsets of the underlying set.
		Take two vertices $u,v$ such that $d(u,v)=k$ and a shortest path $P$ from $u$ to $v$.
		Each ball of radius at most $\frac{t}{2^{m+1}}$ covers at most
		$\frac{t}{2^{m}}+1$ vertices of $P$. Hence $m-1$ such balls cover at most
		$(m-1) (\frac{t}{2^{m}}+1)$ vertices of $P$. Since $t=k/2$, we have that
			\[(m-1) (\frac{t}{2^{m}}+1)=(m-1)\big(\frac{k}{2^{m+1}}+1\big)
			= k\frac{m}{2^{m+1}}-\frac{k}{2^{m+1}}+m-1.\]

		Using  $\frac{m}{2^{m+1}}<\frac12$ and $-\frac{k}{2^{m+1}}\leq 0$ and
		$m\leq \lfloor \log (k/2) \rfloor$,
		we obtain
			\[k\frac{m}{2^{m+1}}-\frac{k}{2^{m+1}}+m-1 <\frac{k}{2}+\lfloor \log (k/2) \rfloor-1.\]
		Hence by the displayed lines we obtain
		\[(m-1) (\frac{t}{2^{m}}+1) <\frac{k}{2}+\lfloor \log (k/2) \rfloor-1
		\leq \frac{k}{2}+\log k - 1 \leq k.\]
		Hence at least one of the $k+1$ vertices of $P$ is not covered.
	\end{proof}
\end{lemma}
Lemma \ref{NoCover} opens up a winning strategy for  Duplicator in Ehrenfeucht-Fraïssé games between two Johnson graphs.
The strategy relies on keeping track of distances between played vertices, but ignoring those  pairs of vertices that are too far apart to feasibly interact.

\begin{theorem}
	\label{ThmEF}
	Duplicator has a winning strategy in  the Ehrenfeucht-Fraissé game
	$EF_{\floor{\log{t}}}(J(n,k),J(n',k'))$
	where
	$t=\floor{\frac{\min(k,n-k,k',n'-k')}{2}}$.
	\begin{proof}
		Since $J(n,k)\cong J(n,n-k)$ we can assume without loss  of generality $n≥2k$ and $n'≥2k'$.
		Moreover let us assume that $k'≤k$.
		This means that $t=\floor{\frac{k'}{2}}$.

		As a clarification on notation 	we will denote by $a_i$ the vertex played on turn $i$ in $J(n,k)$,   regardless of which player plays it.
		Likewise 			we will denote by $b_i$ the vertex played on turn $i$ in $J(n',k')$, regardless of which player plays it.

		Duplicator will be building up a partial bijection $f$ from $X$, the underlying set of $J(n,k)$, to $Y$ the underlying set of $J(n',k')$.
		We do this by creating a sequence $(f_i)_{i=1}^{\floor{\log{t}}}$ where $f_1$ is the empty map, $f_{\floor{\log{t}}}=f$ and  $f_{i+1}$ extends $f_{i}$.
		It is through this sequence of functions that duplicator is able to avoid any traps set by spoiler and is guaranteed a move on every turn of the game.

		\paragraph{Base case:}

		By vertex transitivity, the first move of the game is arbitrary both by spoiler and duplicator.
		Duplicator sets $f_1$ to be the empty map.


		\paragraph{Induction hypothesis:}
		After $m-1$ moves by both players which grant us the sequences $(a_i)_{i=1}^{m-1}, (b_i)_{i=1}^{m-1}$ in $J(n,k)$ and $J(n',k')$ respectively,  duplicator has formed the function
		$f_{m-1}$ in such a way that, for each $i,j\leq m-1$, if
		$|a_i\symdiff a_j|≤\frac{t}{2^{m-1}}$ or $|b_i\symdiff b_j|≤\frac{t}{2^{m-1}}$  then 
		$a_i\symdiff a_j⊆ \dom(f_{m-1})$ , $b_i \symdiff b_j ⊆ \img(f_{m-1})$ and
		$$x∈a_i∖ a_j \iff f_{m-1}(x)∈b_i∖ b_j$$
		Note that this implicitly means that for all $i,j<m$ if $d(a_i,a_j)≤ \frac{t}{2^m}$ or $d(b_i,b_j)≤\frac{t}{2^m}$ then  $d(a_i,a_j)=d(b_i,b_j)$.

		\paragraph{Inductive step:}
		We consider separately how to respond to spoiler playing $a_m$ or $b_m$.

		\paragraph{Assume spoiler plays $a_m$.}
		Let $M⊆[1,m-1]∩ℕ$ be such that $i∈M$ if and only if $d(a_i,a_m)≤\frac{t}{2^{m+1}}$.

		If $M=∅$ then duplicator sets $f_m=f_{m-1}$ and plays a $b_m$ such that $d(b_i,b_m)>\frac{t}{2^{m+1}}$ for all $i<m$.
		Such a vertex always exists due to Lemma \ref{NoCover}.

		Assume $M≠∅$.
		Duplicator has to extend $f_{m-1}$ to $f_m$ with domain
		$$\dom(f_{m-1})∪⋃_{i∈M} a_m\symdiff a_i$$  in such a way that if we have for some $i∈M$ that $x∈a_i∩ \dom(f_{m})$ then $f_{m}(x)∈b_i$ and every element in $( a_m∖(⋃_{i<m}a_i) ) ∩ \dom(f_m)$
			gets mapped to an element that is  in the  complement  of $⋃_{i∈M}b_i$.

		For every $i,j∈M$ we have that $d(a_i,a_j)≤ \frac{2t}{2^{m+1}}$ so by induction hypothesis $a_i\symdiff a_j⊆\dom(f_{m-1})$.
		So
		$$(⋃_{i∈M} a_m\symdiff a_i) ∖ \dom(f_{m-1})  ⊆  (a_m ∖ ⋃_{i∈M}a_i) ∪ ((⋂_{i∈M}a_i)∖a_m)$$
		that is to say, every element that duplicator is adding to  the domain of $f_m$ is either in only $a_m$  or in every other $a_i$ for $i∈M$.
		So the elements that will be in the domain of $f_m$ that are not in the domain of $f_{m-1}$ will be in either
		$(⋂_{i∈M} a_i) ∖ a_m$
		or
		$a_m∖⋃_{i∈M} a_i$
		and thus need to be mapped to
		$(⋂_{i∈M} b_i)$
		and
		$Y∖(⋃_{i∈M} b_i)$ respectively.
		In each case, to extend $f_{m-1}$ to $f_{m}$, we need to show that the domain contains at most as many elements as the image.

		\paragraph{Case 1:}
		Mapping
		$(⋂_{i∈M} a_i) ∖ a_m$
		to
		$(⋂_{i∈M} b_i)$
		is possible since:
		We know that for all $i∈M$ $|a_i∖a_m|≤ \frac{t}{2^{m+1}}$ so we get:
		$$|(⋂_{i∈M} a_i) ∖ a_m| ≤ \frac{t}{2^{m+1}} ≤  t$$
		and
		$$|(⋂_{i∈M} b_i)|≥ k'- \frac{|M|t}{2^{m}} ≥ k' - \frac{tm}{2^{m}} ≥ k' - t ≥ t.$$
		For the first inequality, observe that if $i,j\in M$, then $|b_i-b_j|=|a_i-a_j|\leq  \frac{t}{2^{m}}.$
		\paragraph{Case 2:}
		We can map
		$a_m∖⋃_{i∈M} a_i$
		to
		$Y∖(⋃_{i∈M} b_i)$ due to the following:
		Since for $i∈M$ $|a_m∖a_i|≤ \frac{t}{2^{m+1}}$ we get
		$$|a_m∖⋃_{i∈M} a_i| ≤ \frac{t}{2^{m+1}}≤ t$$ 
		and
		$$|Y∖(⋃_{i∈M} b_i)| ≥  n' - (k' + \frac{|M|t}{2^{m}}) ≥ (n' - k') - m\frac{t}{2^{m}} ≥ k' - \frac{mt}{2^{m}} ≥ k'-t ≥ t$$

		In either case there are sufficiently many unmapped elemnts for duplicator to extend $f_{m-1}$ to $f_m$.
		Duplicator then plays $b_m=f_m[a_m∩ \dom(f_m)] ∪ ⋂_{i∈M}b_i$.
		Note that the set~$b_m$ is a vertex, i.\,e.\ a set of size $k'$, by the definition of $f_m$.

		Also then we have for all $i∈M$ that  $d(a_m,a_i)=d(b_m,b_i)$.
		So by induction hypothesis  we have for $i,j≤m$ that if
		$|a_i\symdiff a_j|≤\frac{t}{2^{m}}$ or $|b_i\symdiff b_j|≤\frac{t}{2^{m}}$  then
		$a_i\symdiff a_j⊆ \dom(f_{m})$ , $b_i \symdiff b_j ⊆ \img(f_{m})$ and
		$$x∈a_i∖ a_j \iff f_{m}(x)∈b_i∖ b_j.$$

		\paragraph{Assume spoiler plays  $b_m$.}
			A similar argument to above gives duplicator a suitable $a_m$ to play.
		So after $m$ moves by both players, duplicator has formed the function $f_{m}$ in such a way that if
		$|a_i\symdiff a_j|≤\frac{t}{2^{m}}$  then  $x∈a_i ∖ a_j \iff f_m(x)∈b_i∖b_j$ and $x∈a_j∖a_i \iff f_{m}(x)∈b_j∖b_i$. 
		This concludes the inductive step.

		At the end of the game we have for any $i,j≤\floor{\log{t}}$
		that if
		$d(a_i,a_j)≤\floor{\frac{t}{2^{\log{t}}}}=1$ or
		$d(b_i,b_j)≤\floor{\frac{t}{2^{\log{t}}}}=1$
		then
		$d(a_i,a_j)=d(b_i,b_j)$.
		In particular this means that $Ea_ia_j$ if and only if $Eb_ib_j$ so the substructures
		$A=\{a_i\colon 1≤i≤\log(t)\}$ in $J(n,k)$
		and
		$B=\{b_i\colon 1≤i≤\log(t)\}$ in $J(n',k')$ are isomorphic with isomorphism $a_i\mapsto b_i$.
	\end{proof}
\end{theorem}

\begin{theorem}
\label{InfiniteEF}
Let $(κ_i)_{i∈I}$ and $(λ_i)_{i∈I}$ be families of infinite cardinals.
Duplicator has a winning strategy in  the Ehrenfeucht-Fraïssé game
$EF_{\floor{\log{t}}}(J(n,k),\bigsqcup_{i∈I} J'(κ_i,λ_i))$
where $t=\floor{\frac{\min(k,n-k)}{2}}$.
\begin{proof}

Duplicator's strategy in the proof of Theorem \ref{ThmEF} does not make any distinction between vertices further than $\frac{t}{2^m}$ from played vertices on every step of play.
So vertices in separate connected components are simply treated as being too far from each other to interact and Duplicator can apply the same strategy.
\end{proof}
\end{theorem}

In this section we have shown that the $J'(κ,λ)$ are a natural extension of Johnson graphs into infinite graphs.
Furthermore we have shown that {\GenPhrase}s live up to their name and generalize the notion of infinite Johnson graphs while remaining elementarily equivalent.
Theorem \ref{ThmEF} together with Theorem \ref{InfiniteEF} gives us Theorem \ref{EF-Intro}.
It follows that as $n$, $k$, and $n-k$ all grow, the truth value of any sentence in the language of graphs eventually stabilizes for $J(n,k)$ and therefore we have a limit theory $T_{\Johnson}$ for Johnson graphs, which is a complete theory, and simultaneously the theory for {\GenPhrase}s, thus proving Corollary \ref{CompleteTheory}
\section{Axiomatization}
\renewcommand*{\MyPath}{Axioms}
\subsection{Axiomatic theory of Johnson graphs}
\label{AxChap}
In this section we will give axioms for $T_{\Johnson}$.
We will first give the axioms that are satisfied by any Johnson graph.
We then give the axiom schemas that give information about parameters $n$ and $k$ for $J(n,k)$.
	We are primarily interested in the axiom schemas where $n$ and $k$ are left unbounded.
We will then show that a connected graph $M$ is isomorphic to $J'(κ,λ)$ for some $κ$ and $λ$ if and only if it satisfies those axioms.
Before we begin we introduce the shorthands we use in the definitions below.
\begin{definition}
Let $L$ be the language of graphs i.e. a language with a single binary relation symbol $E$.\\
Let $L_c$ be the language of graphs with $3$ additional constant symbols $v,a,$ and $b$.
\end{definition}

We will first give our axioms in $L_c$.
This is done for convenience, and allows us to give the axioms in an intuitive manner that describes what Johnson graphs look like locally.
We will later on argue that the axioms apply for any $3$ pairwise adjacent vertices so we interpret the constant symbols of $L_c$ as variables in $L$ and obtain axioms for the theory of Johnson graphs.

\subsubsection{Definable language extensions}

%
%
%
%
		We use $Set_n((x_i)_{i=1}^n)$ to mean $(x_i)_{i=1}^n$ is a set of size $n$, i.\,e.\ they are $n$ distinct elements.
		We use
		$\{x_i\}_{i=1}^n =\{y_i\}_{i=1}^n$
	to mean that
		$(x_i)_{i=1}^n$ and $(y_i)_{i=1}^n$ are sets, i.\,e.~they each have $n$ distinct elements, and they represent the same set, i.\,e.\ they are equal up to permutation of indices.
%
%
%

	We write $d_n(x,y)$ to mean "$x$ is at distance $n$ from $y$"
	and for the special vertex $v$ we let
	$D_n:=\{u:d_n(v,u)\}$.
In addition to these we show that the underlying set of a Johnson graph can be expressed and argued with in our language.
In order to do that we first need to describe the local structure of our graph.
	\newcommand{\Rep}{\operatorname{Rep}}
	\newcommand{\RRep}{\operatorname{RRep}}
	\newcommand{\CRep}{\operatorname{CRep}}
	\newcommand{\Col}{\operatorname{Col}}
	\newcommand{\Row}{\operatorname{Row}}
	\paragraph{Rook's graph structure}
	Lemma \ref{nRook} gives us that $N(v)$ induces a rook's graph as a subgraph.
	We work with the intuitive notions of row and column in the rook's graph.
	We use the constants $a$ and $b$ to identify the difference  between two adjacent vertices in $N(v)$ sharing a row on one hand and sharing a column on the other.
	We say $a$ and $b$ are to be in the same column and  the members of that column are representatives of their respective rows.
	The row that $a$ belongs to is used as the set of representatives for columns.
	$$\Rep(x):= Evx ∧ (Exa ∨ x=a)$$
	$$\RRep(x):= Rep(x) ∧ (Exb ∨ x=b)$$
	$$\CRep(x):= Rep(x) ∧ (¬\RRep(x) ∨ x=a)$$

	We then say that two vertices belong to the same row (column) if they are adjacent to the same  row (column) representative.

	$$\Col(x,y) := (Evx ∧ Evy) ∧ (Exy ∨ x=y) ∧ ∃z (\CRep(z) ∧ (Exz ∨ x=z) ∧ (Eyz ∨ y=z))$$
	$$\Row(x,y) := (Evx ∧ Evy) ∧ (Exy ∨ x=y) ∧ ∃z (\RRep(z) ∧ (Exz ∨ x=z) ∧ (Eyz ∨ y=z))$$

	\paragraph{Underlying sets}
	We will now give formulae in  $L_c$  that capture the underlying set structure of the Johnson graph.
	Recall that the vertices correspond to subsets of the underlying set of the Johnson graph.
	We then use the fact that in the rook's graph on the neighbours of the vertex $v$ in a Johnson graph,
	each column contains those vertices that are missing the same element of $v$
	and each row those that include the same element from outside $v$.
	We let
	\begin{align*}
		ϕ_1(x,u,w)&:= \CRep(u)∧\RRep(w) ∧\\
			& ((¬\Rep(x)∧Exu ∧ Exw ∧ Exv) ∨ x=u=w=a ∨\\
			& (\CRep(x)∧x=u∧Exw)∨ (\RRep(x)∧Exu∧x=w)).\\
	\end{align*}
	Thus $ϕ_1(x,u,w)$ holds if and only if $x$ is a neighbour of $v$ in the column with $u$ for a representative and the row with  $w$ for a representative. For the vertices at distance $n$ from $v$ we define
	\begin{align*}
		ϕ_n(x,(u_i)_{i=1}^n,(w_i)_{i=1}^n)&:=d_n(v,x) ∧ Set_n((u_i)_{i=1}^n) ∧ Set_n((w_i)_{i=1}^n) ∧\\ &\bigwedge_{i=n}^n (\CRep(u_i) ∧ \RRep(w_i)) ∧\\
						& ∀y (d_{n-1}(v,y) \to (Exy \leftrightarrow
							( \bigvee_{i=1}^n \bigvee_{j=1}^nϕ_{n-1}(y,(u_{i'})_{\substack{i'=1\\i'≠i}}^n,(w_{j'})_{\substack{j'=1\\j'≠j}}^n)) )).
 	\end{align*}
		We want $ϕ_n$ to be a formula determining a bijective function $D_n \to \binom{\CRep}{n} \times \binom{\RRep}{n}$.
		We call that function $f_n \times g_n$ where $f_n(x):= \{u_i:1≤i≤n\}$ and $g_n(x):=\{w_i:1≤i≤n\}$.
		Intuitively, vertex $x$ represents the subset $(v∖f_n(x))∪(g_n(x)∖v)$ of the underlying set,
		i.\,e.\ for all vertices $x$ we have
		$$ ϕ_n((x,f_n(x),g_n(x)).$$

\subsubsection{Common Axioms}
	We start with the basic axioms of simple undirected graphs:
	\begin{align}
		\label{GraphAx1}
		&∀x¬Exx			\\
		&∀x∀y (Exy \to Eyx)
		\label{GraphAx2}
	\end{align}
	We also need an axiom stating that our choice of constants $v,a,$ and $b$ form a clique.
	\begin{align}
		\label{ConstAx}
		Eva∧Evb∧Eab
	\end{align}
	Next we will give the additional axioms that give us the rook's graph structure of $N(v)$.
	First, every vertex in $N(v)$ that is not a representative is adjacent to a unique row representative, and a unique column representative.
		\begin{align}
		\label{CRepUniqueAx}
			&∀x \big(d_1(v,x)∧¬\Rep(x) \to (\exists! y (\CRep(y)∧Exy)) \big)\\
		\label{RRepUniqueAx}
		&∀x \big(d_1(v,x)∧¬\Rep(x) \to (\exists! y (\RRep(y)∧Exy))\big)
		\end{align}
Here $\exists !y$ is read as `there is a unique $y$'.

	The relation $\Row$ is an equivalence relation on the set of neighbours of $v$:
	\begin{align}
		\label{RowIsEquiv}
		&∀x (d_1(v,x) \to  \Row(x,x))			\\
		&∀x∀y (\Row(x,y) \to \Row(y,x))			\\
		&∀x∀y∀z ((\Row(x,y)∧\Row(y,z)) \to \Row(x,z))
	\end{align}
	The relation $\Col$ is an equivalence relation on the set of neighbours of $v$:
	\begin{align}
		&∀x (d_1(v,x) \to \Col(x,x))			\\
		&∀x∀y (\Col(x,y) \to \Col(y,x))		\\
		&∀x∀y∀z ((\Col(x,y)∧\Col(y,z)) \to \Col(x,z))
		\label{ColIsEquiv}
	\end{align}
	Adjacent vertices share a row or a column:
		\begin{align}
			\label{ERowCol}
			∀x∀y  \big[(d_1(v,x) ∧ d_1(v,y)) \to \big((Exy ∨ x=y) \leftrightarrow (\Col(x,y) ∨ \Row(x,y)) \big)\big]
		\end{align}
	Two distinct vertices can share a row or column but never both:
		\begin{align}
			\label{RowOrCol}
			∀x∀y ((\Col(x,y) ∧ \Row(x,y)) \to x=y)
		\end{align}
	Every row intersects with every column in a singleton.
	This guarantees that all rows are equinumerous and all columns are equinumerous and they interact in the desired way.
		\begin{multline}
			\label{OppCorner}
			∀x∀y \big((d_1(v,x) ∧ d_1(v,y) ∧ ¬Exy ∧ x≠y) \to (∃!u(\Row(x,u)∧\Col(y,u)))\big)
		\end{multline}

	The above axioms determine that the neighbourhood of $v$ is a rook's graph.
	We now want to give axiom schemas to characterize vertices further away from $v$.
	For this purpose we will be relying on the functions $f_n$ and $g_n$ as defined by $ϕ_n$.

		\begin{align}
		\label{BiAx}
		\text{The function } f_n\times g_n:=  D_n \to \binom{\CRep}{n} \times \binom{\RRep}{n}\text{ is bijective.}
		\end{align}

		In $L_c$ we can express this as:
		\begin{align*}
			∀x \big(&
				d_n(v,x)
				\to (
					∃(u_i)_{i=1}^n ∃(w_i)_{i=1}^n (
					\bigwedge_{i=1}^n (\CRep(u_i)∧\RRep(w_i)) ∧
					Set_n((u_i)_{i=1}^n) ∧ Set_n((w_i)_{i=1}^n)\\&
∧
						ϕ_n(x,(u_i)_{i=1}^n,(w_i)_{i=1}^n)
						∧
						∀(u'_i)_{i=1}^n ∀(w'_i)_{i=1}^n (
							ϕ_n(x,(u'_i)_{i=1}^n,(w'_i)_{i=1}^n) \to \\& (\{u_i\}_{i=1}^n=\{u'_i\}_{i=1}^n ∧ \{(w_i)_{i=1}^n\}=\{(w'_i)_{i=1}^n\})
						)
					)
				)
			\big) ∧&\\
			∀(u_i) ∀(w_i) \big(&
				(
					\bigwedge_{i=1}^n (\CRep(u_i)∧\RRep(w_i)) ∧
					Set_n((u_i)_{i=1}^n) ∧ Set_n((w_i)_{i=1}^n)
				)
				\to (
					∃x ( d_n(v,x)  \\&
						ϕ_n(x,(u_i)_{i=1}^n,(w_i)_{i=1}^n)
						∧
						∀y (
							ϕ_n(y,(u_i)_{i=1}^n,(w_i)_{i=1}^n) \to (x=y)
						)
					)
				)
			\big)
		\end{align*}

%
%
%
		Let $h_n(x):=f_n(x)∪g_n(x)$.
		When $x$ and $y$ are both at  distance $n$ from $v$ then $Exy$ if and only if
			$|h_n(x) \symdiff h_n(y)|=2$.
%
		In $L_c$ we can express this by
		\begin{align*}
			∀x∀y&\big(( d_n(v,x) ∧ d_n(v,y))  \to
				(
				 Exy  \leftrightarrow
					(
						 ∃u ∃u' ∃(u_i)_{i=2}^n ∃w ∃w' ∃(w_i)_{i=2}^n \\
						&(
						 (w=w'\leftrightarrow u≠u')∧\\
						&(
							ϕ_n(x,u ,(u_i)_{i=2}^n,w ,(w_i)_{i=2}^n)
							∧
							ϕ_n(y,u',(u_i)_{i=2}^n,w',(w_i)_{i=2}^n)
						)
						)
					)
				)
			\big),
		\end{align*}
		but for simplicity we will write it as
		\begin{align}
			\label{Enn}
			∀x∀y(( d_n(v,x) ∧ d_n(v,y))  \to (Exy  \leftrightarrow|h_n(x)\symdiff h_n(y)|=2)).
		\end{align}

		The above condition fully specifies the vertices at distance $n$, hence  presenting them for each $n$ inductively will give us a full axiomatization of the radius $n$ balls in Johnson graphs.

\subsubsection{Size axioms}
Here we give the axiom schemas required to specify  the parameters $m$ and $k$ for the Johnson graph $J(m,k)$.
As we are primarily interested in the case where $m$,$k$ and $m-k$ are all unbounded we write these schemas in such a way that each axiom only specifies a lower bound for $m$ and $k$.
We do this by giving a lower bound on the number of vertices that satisfy $\RRep$ and $\CRep$.

For $J(m,k)$ we say:

\begin{align}
	\label{RSize}		σ_{\ell}^{\Row}:=(∃x_i)&_{i=1}^{\ell}\big(⋀_{1\leq i<j\leq\ell}  x_i≠x_j ∧  ⋀_{i=1}^{\ell} \RRep(x_i)\big), \text{ for }\ell=m-k  \\
\label{CSize}		σ_{  k}^{\Col}:=(∃x_i)&_{i=1}^{  k} \big(  ⋀_{1\leq i<j\leq k}  x_i≠x_j ∧
⋀_{i=1}^{  k} \CRep(x_i)\big)\end{align}
\TODO[I strongly think the previous version was better]
These axioms together with  (\ref{BiAx}), give us the number of vertices in the ball of radius $n$ around $v$.
Note that for $\binom{m}{k}<3$ the Johnson graphs $J(m,k)$ do not satisfy the above axioms as our assumption that the constant symbols $v,a$ and $b$ can be interpreted in our model is false in those cases.
This only applies to a couple of very small models whose theory can easily be described by specifying the interpretation of $E$.

We claim that the $L_c$-theory of $J'(κ,λ)$  is axiomatized by the above axioms leaving $m$, $k$ and $m-k$ unbounded.

\begin{definition}
	Let $Σ_c$ be the set of $L_c$-axioms (\ref{GraphAx1}) through (\ref{CSize}), noting that \ref{BiAx} through (\ref{CSize})
	are axiom schemes, where $σ_{\ell}^{Row}$ and $σ_k^{Col}$ hold for all $\ell$ and $k$, respectively.
\end{definition}

We write $J'(κ,λ)_c$
for $J'(κ,λ)$ expanded by constants $v,a,b$ interpreted by pairwise adjacent vertices.

\begin{theorem}
\label{AxThm}
Let $κ$ and $λ$ be infinite cardinals and $M$ be a connected graph viewed as an $L_c$-structure for some interpretation of $v,a,b$.
Then  $M\models Σ_c$ if and only if $M\cong J'(κ,λ)_c$.
\end{theorem}
\begin{proof}
First we show that $J'(κ,λ)_c\models Σ_c$.
	We then show that a connected model that satisfies the above axioms fits our definition of $J'(κ,λ)_c$.


Assume we have  $J'(κ,λ)$ with underlying set $X$ and central vertex $A$.
First we choose vertices for our constant symbols $v,a,b$.
We let $v=A$ and $a,b$ be any two neighbours of $A$  such that
$a=(A∖\{x\})∪\{y\}$ and  $b=(A∖\{x\})∪\{z\}$ where $x∈A$ and $y,z∈X∖A$ with $y≠z$.
We quickly note that this interpretation of $v,a,b$ guarantees that axiom (\ref{ConstAx}) holds.

Second we want to show  how the formulae defining
$\Rep,\CRep,\RRep,\Row,\Col$ and $d_n$ will be interpreted.
We have for any $n∈ℕ$ and vertices $v_1,v_2$ that $d_n(v_1,v_2)$ holds if $v_1$ and $v_2$ are at distance $n$ from each other and $D_n$ is the set of  those vertices that are at distance $n$ from $v$.
Note that since $J'(κ,λ)$ is a connected graph we have for any vertex $u$ that there exists an $i∈ℕ$ such that $d_i(v,u)$ holds.

We now observe that the sets of vertices satisfying
$\Rep,\RRep,\CRep$ are
\begin{align*}
&\{(A∖\{i\})∪\{j\}:i=x ∨ j=y\}\\
&\{(A∖\{x\})∪\{j\}:j∈X∖A\} \\
&\{(A∖\{i\})∪\{y\}:i∈A\}
\end{align*}
respectively.
Note that each of these sets has size at least $ω$ so (\ref{RSize}) and (\ref{CSize}) hold.
Furthermore, every neighbour of $v$ that is not in $\Rep$ has a unique neighbour in $\CRep$, and a unique neighbour in $\RRep$, so  (\ref{CRepUniqueAx}) and  (\ref{RRepUniqueAx}) hold.

This gives us that
$\Row((A∖\{i\})∪\{l\},(A∖\{j\})∪\{k\}) \iff l=k$
and
$\Col((A∖\{i\})∪\{l\},(A∖\{j\})∪\{k\}) \iff i=j$.
Note that this means that $\Row(v_1,v_2)$ and $\Col(v_1,v_2)$ align with our intuition of $v_1$ and $v_2$ sharing a row and column in the rook's graph respectively.

We can now go through the remaining axioms in $Σ_c$ and show that $J'(κ,λ)_c$ satisfies each of them with the above interpretations.
$J'(κ,λ)_c$ is a graph so it satisfies (\ref{GraphAx1}) and (\ref{GraphAx2}).
	By the interpretation of $c,a$ and $b$, Axiom~\ref{ConstAx} holds.
Since $\Row$ and $\Col$ are equivalence relations on the rook's graph induced by the neigbourhood of $v$, $J'(κ,λ)_c$ satisfies (\ref{RowIsEquiv}) through (\ref{ColIsEquiv}).

Likewise since the rook's graph has all edges within any given row and all edges within any given column  and no other edges (\ref{ERowCol}) holds in $J'(κ,λ)_c$.

Since a row and a column always meet in exactly one vertex we have $J'(κ,λ)_c$ satisfies (\ref{RowOrCol}) and (\ref{OppCorner}).

If we take any two such vertices
$(A∖\{i\})∪\{l\}$ and $(A∖\{j\})∪\{m\}$ with $i≠j$ and $l≠m$ we have that they have exactly two neighbours in $D_1$ in common,
namely
$(A∖\{i\})∪\{m\}$ and $(A∖\{j\})∪\{l\}$.
So $J'(κ,λ)_c$ satisfies (\ref{OppCorner}).

Finally we check that the axiom schemas hold for $J'(κ,λ)_c$.

We write any vertex of $J'(κ,λ)_c$ at distance $n$ from $v$ in the form $(A∖B)∪C$ where
$B⊆A$, $C⊆X∖A$ and $|B|=n=|C|$.
Recall that $f_n$ and $g_n$ are components of a  bijective function  $D_n \to \binom{\CRep}{n}\times \binom{\RRep}{n}$.
We want to use $a=(A∖\{x\})∪\{y\}$ to fix the orientation and say that
$$f_n((A∖B)∪C)=		\{(A∖\{i\})∪\{y\}:i∈B\}
$$ and
$$g_n((A∖B)∪C)= \{(A∖\{x\})∪\{j\}:j∈C\}
$$

We note that for any $n$,  $f_n\times g_n$ is a bijection $D_n \to \binom{\CRep}{n} \times \binom{\RRep}{n}$ so the schema  (\ref{BiAx}) holds.
We also note that if we have $2$  vertices
$(A∖B)∪C$ and $(A∖D)∪F$ adjacent to each other then
$|((A∖B)∪C)\symdiff ((A∖D)∪F)|= 2$ which implies
	$|(B∪C)\symdiff (D∪F)|=2$.
If both are at distance $n$ from $v$ then $|B|=|C|=|D|=|F|=n$
 which gives us 
	$|h_n((A∖B)∪C)\symdiff h_n((A∖D)∪F)|=2$ so the schema (\ref{Enn}) holds.


We have now shown how each of the axioms in $Σ_c$  is satisfied and conclude that $J'(κ,λ)_c\models Σ_c$.

Conversely, assume $M \models Σ_c$ connected.
We want to show that $M$ is isomorphic to $J'(κ,λ)_c$ for some $κ$ and $λ$.
To do this we will first identify the underlying set  $X$  of $J'(κ,λ)_c$ with a certain set interpretable in $M$.
We will then  construct an isomorphism $h$ from $M$ to $J'(κ,λ)_c$ as $L_c$-structures.

Since $M$ satisfies the axioms of graph theory $M$ is a graph with $E$ as the edge relation.
We check the cardinality of the sets of representatives, that is of both
${\{x∈M\colon \RRep(x)\}}$ and
${\{x∈M\colon \CRep(x)\}}$.
By (\ref{RSize}) and (\ref{CSize}) we know that
$|\{x∈M\colon \RRep(x)\}|≥ω$ and
$|\{x∈M\colon \CRep(x)\}|≥ω$.
If
$|\{x∈M:\RRep(x)\}| ≤ |\{x∈M:\CRep(x)\}|$
we say
$|\{x∈M:\RRep(x)\}|=λ$ and $|\{x∈M:\CRep(x)\}|=κ$,  otherwise we say
$|\{x∈M:\RRep(x)\}|=κ$ and $|\{x∈M:\CRep(x)\}|=λ$.
We assume
$|\{x∈M:\CRep(x)\}| ≤ |\{x∈M:\RRep(x)\}|$.
This assumption can not a priori be taken without loss of generality,
but the argument when
$|\{x∈M:\CRep(x)\}| > |\{x∈M:\RRep(x)\}|$ is identical with the roles of the two sets reversed.
This is justified a posteriori by the observation that it is analogous to swapping $J(n,k)$ and $J(n,n-k)$.
In fact the graph $J'(κ,λ)$ with underlying set $X$ is isomorphic to the graph obtained by replacing each vertex $A$ with $X∖A$ and having two vertices adjacent if and only if their symmetric difference has size $2$.

We have $D_1:=D_1(v)$, and by axioms (\ref{RowIsEquiv}) through (\ref{ColIsEquiv}) we know that $\Col$ and $\Row$ are equivalence relations on $D_1$.
We can therefore take the quotients
 of those relations (i.e. the sets of their respective equivalence classes) and define $A:=D_1/\Col$  and $B:=D_1/\Row$.

Our definition of the relations $\Row$ and $\Col$ and the axioms stating that they are reflexive on the neighbours of $v$ gives us that every neighbour $u$ of $v$ has both a row- and column-representative denoted $r_u$ and $c_u$ respectively.
We then define $X:=A∪B$ and construct an isomorphism $h:M\to J'(κ,λ)_c$ as follows.
\begin{align*}
		&h(v):=A⊂X\\
∀u∈D_1		&h(u):=(A∪\{r_u/\Row\})∖\{c_u/\Col\}\\
\end{align*}


Given $x∈D_n$, define	$f_n^*(x):=\{u/\Col:u∈ f_n(x)\}$ and
			$g_n^*(x):=\{u/\Row:u∈ g_n(x)\}$.
 			$h(x):=(A ∪ g_n^*(x))∖ f_n^*(x)$.
We have now defined a full function  $h$ from  $M$ to $J'(κ,λ)_c$.
We note that $h$ is a bijection by our choice of $A$, $B$ and (\ref{BiAx}).
We now just need to check that it preserves the edge relation to prove that it is an isomorphism.

Consider two adjacent vertices $u$ and $w$ at distance $n$ from $v$.
Then by Axiom (\ref{Enn}) we have that
$|(f_n(u)\symdiff f_n(w)) ∪ (g_n(u)\symdiff g_n(w))|=2$.
This gives us that:
$$|(A∖f_n^*(u))∪(g_n^*(u)∖A) \symdiff (A∖f_n^*(w))∪(g_n^*(w)∖A)|=2$$
So the sets corresponding to $u$ and $w$ are adjacent in $J'(κ,λ)_c$.

Also consider two adjacent vertices $u$ and $w$ at distance $n$ and $n-1$ respectively  from $v$.
Then by Axiom (\ref{Enn}) $n$ we have that $f_{n-1}(w)⊂ f_n(u)$.
Hence,
$$|((A∖f_n^*(u))∪g_n^*(u)) \symdiff ((A∖f_{n-1}^*(w))∪g_{n-1}^*(w))|=2$$
Again, the sets corresponding to $u$ and $w$ are adjacent in $J'(κ,λ)_c$.

So $h$ is an isomorphism  $M \to  J'(κ,λ)_c$.

\end{proof}

So we know $Σ_c$ describes Johnson graphs locally around a vertex $v$.
We now view  $Σ_c$ as a set of formulae $ϕ(v,a,b)$ in $L$, the language of graphs, by taking  $v,a,$ and $b$  to be variables. 
Then if we bind each variable with a quantifier we obtain a  set of sentences in $L$ rather than $L_c$.
\begin{definition}
$Σ:= \{∀v∀a∀b ((Eva∧Evb∧Eab) \to ϕ(v,a,b)): ϕ∈Σ_c \} ∪ \{∀x∃y∃z (Exy ∧ Exz ∧ Eyz)\}$ is a set of sentences in $L$.
\end{definition}
For simple referencing we give a number to the new axiom.
\begin{align}
	∀x∃y∃z (Exy ∧ Exz ∧ Eyz)
	\label{LastAx}
\end{align}

\begin{theorem}{}
\label{LAxThm}
Let $M$ be a connected model and $κ$ and $λ$ be infinite cardinals.
Then $M\modelsΣ$ if and only if $M\cong J'(κ,λ)$.
\begin{proof}
First assume $M\cong J'(κ,λ)$.
We know by Theorem \ref{AxThm} that there exist $v,a,b$ pairwise adjacent such that $ϕ(v,a,b)$ holds for all $ϕ∈Σ_c$.
By Theorem \ref{InfiniteEF} we have that any two $J'$ with infinite parameters are elementarily equivalent, so we can assume $κ=λ$.
This allows us to have automorphisms that interchange rows and columns.
Then we have for any triangle  $v',a',b'∈M$ that there exists an automorphism $F$ of $M$ such that $F(v)=v'$, $F(a)=a'$ and $F(b)=b'$.
So $M\models ϕ(v',a',b')$ for all $ϕ∈Σ_c$.
For any vertex  $u$ in $J'(κ,λ)$,  let  $x∈u$ and  $y,z∈X∖u$  with $(y≠z)$.
	We have that $u∖\{x\}∪\{y\}$ and $u∖\{x\}∪\{z\}$ are two adjacent neighbours of $u$ so (\ref{LastAx}) holds in $J'(κ,λ)$.
Thus $M\models Σ$.

Now assume that $M$ is a connected model such that $M\models Σ$.
Take $v∈M$. Then (\ref{LastAx}) gives us that $v$ has two adjacent neighbours $a$ and $b$.
Then by taking $v,a,b$ as witnesses for the formulae in $Σ$, \ref{AxThm} gives us that $M\cong J'(κ,λ)$.
\end{proof}
\end{theorem}

\begin{corollary}
	\label{GenJohnsonThm}
	$M\models Σ$ if and only if $M$ is a {\GenPhrase}.
	\begin{proof}
	First assume that $M\models Σ$. Then (\ref{LastAx}) guarantees that every connected component has a triangle and therefore by Theorem \ref{LAxThm} we have that each connected component is isomorphic to $J'(κ,λ)$ for some infinite $κ$ and $λ$. Thus $M$ satisfies the definition of {\GenPhrase}.
	Conversly if we assume $M$ is a {\GenPhrase}, then by Theorems \ref{InfiniteEF} and  \ref{LAxThm} we get that $M\models Σ$.
	\end{proof}
\end{corollary}
\begin{corollary}
	\label{TJAxThm}
	$Σ$ is a complete set of axioms for $T_{\Johnson}$.
	\begin{proof}
		This follows directly from  Theorem \ref{InfiniteEF} and Corollary \ref{GenJohnsonThm}.
	\end{proof}
\end{corollary}

We note that Theorem \ref{AxiomResultsIntro} follows from Corollaries \ref{TJAxThm} and \ref{GenJohnsonThm}.
This gives us a nice description of the models of $T_{\Johnson}$, namely that they are exactly the  {\GenPhrase}s.
This is particularly useful since it tells us that an elementary extension of a {\GenPhrase} is also a {\GenPhrase}.
In the next section we will leverage this fact to prove the stability of $T_{\Johnson}$.

\section{Stability}

In this section we prove that $T_\Johnson$ is an $ω$-stable theory of Morley rank $ω$ (Theorem \ref{Stability}).
We first prove that the Morley rank of balls in {\GenPhrase} grows with their radius, giving us a lower bound of $ω$.
The key observation to get the upper bound of $ω$ is that any finite union of balls of finite radius has finite Morley rank.

\begin{lemma}
\label{LongDistanceLemma}
Let $M\models T_{\Johnson}$, and  $ϕ(x,(c_i)_{i=1}^n)$ be a formula with $(c_i)_{i=1}^n ∈M^n$.
Then there is $r=r(ϕ)$ such that for any $d_1$, $d_2∈M$
	both at a distance greater than $r$ from each $c_i$, we have $M\models ϕ(d_1,\bar{c}) \leftrightarrow ϕ(d_2,\bar{c})$.
	\begin{proof} (Sketch)
We will prove this by playing an Ehrenfeucht-Fraïssé game $EF_{\floor{\log(r)}}(M,M)$, where $r$ is chosen sufficiently large.
As in the proof of Theorem \ref{ThmEF}, we will refer to the move made in the left copy of $M$ in the $i$-th round as $a_i$ and the move made in the right copy of $M$ as $b_i$.
Assume that for the first $n$ moves $a_i=b_i=c_i$ and $a_{n+1}=d_1$ and $b_{n+1}=d_2$.
Now duplicator forms a partial function $ f_{n+1} $
	such that for any two  $c_i,c_j$ such that $d(c_i,c_j)<r$ we have that $c_i\symdiff c_j$ is in the domain and image of $f_{n+1}$.
This can be achieved by having $f_{n+1}$ be the identity map everywhere it is defined.
We now note that $(a_i)_{i=1}^{n+1}$, $(b_i)_{i=1}^{n+1}$ and $f_{n+1}$ satisfy an induction hypothesis similar
 to the one used to prove Theorem~\ref{ThmEF}.
Note that $d_1 \symdiff c_i$ and $d_2\symdiff c_i$ are not explicitly added to the domain and image of $f_{n+1}$ for any $i$.
This is due to the fact that if duplicator would be playing the strategy from Theorem \ref{InfiniteEF}
	then the set of played  vertices sufficiently close to $d_1$ and $d_2$ to warrant an update would be empty by our assumption that
	the distance from either $d_1$ or $d_2$ to any $c_i$ is greater than $r$.

So this is a possible intermediate state of the game where the game is played for $\floor{\log(r)}$ rounds.
	So as we have shown in the proof of Theorem \ref{ThmEF} duplicator has a winning strategy.
If we now say that $\floor{\log(r)}$ is at least the quantifier depth of $ϕ$ we know by the Ehrenfeucht–Fraïssé Theorem that 
	$M\models ϕ(d_1,\bar{c}) \leftrightarrow ϕ(d_2,\bar{c})$.
\end{proof}
\end{lemma}
\begin{corollary}
	If $M$ is a {\GenPhrase} then all vertices of $M$ have the same $1$-type.
	\begin{proof}
		Let $d_1,d_2 \in M$. By Lemma~\ref{LongDistanceLemma} (applied without any $c_i$), $d_1$ and $d_2$ have the same type.
	\end{proof}
\end{corollary}

\begin{theorem}
	Let $M$ be a {\GenPhrase}, let $N$ be a non-empty union of connected components of $M$.
	Then $N$ is an elementary substructure of  $M$.
	\begin{proof}
		Let $ϕ$ and $\bar{c}:=(c_i)_{i=1}^n∈N^n$  be  such that $M\models∃x(ϕ(x,\bar{c}))$.
		Let $d$ be a witness i.e. $M\models ϕ(d,\bar{c})$.
		If $d∈N$ we are done.
		Assume $d∉N$. Then by Lemma \ref{LongDistanceLemma} there is an $r_{ϕ}$ such that
			for any $e∈M$ such that $\min(d(e,c_i))≥r_{ϕ}$ we have $M\modelsϕ(d,\bar{c})\iff ϕ(e,\bar{c})$.
		Since $N$ has infinite diameter there is such an $e$ in $N$.
		So by the Tarski-Vaught test $N$ is an elementary substructure of $M$.
	\end{proof}
\end{theorem}

\begin{definition}
	We say that a set of vertices in a {\GenPhrase} is \emph{small} if it is contained in a finite union of balls of finite radius.
	We say that a set is \emph{large} if it is not small.
\end{definition}

\begin{lemma}
	\label{MorleyBall}
	Let $B$ be a ball of radius $r$ in a {\GenPhrase} $M$.
	Then $B$ has Morley rank $2r$.
	\begin{proof}
		Since parameters in a different connected component from  $B$ are irrelevant we can assume that
		$M=J'(κ,λ)$ for some infinite cardinals $κ$ and $λ$.
		This could be seen by an Ehrenfeucht-Fraïssé argument similar to the one in Lemma \ref{LongDistanceLemma}.
		Alternatively, observe that the argument below uses automorphism groups, and that the pointwise
		stabiliser in $\operatorname{Aut}(M')$ of $M'∖ M$ induces the whole of $\operatorname{Aut}(M)$.

		Let $v$ be the center of our ball $B$ and $a,b∈N(v)$ such that $Eab$.
		From Section \ref{AxChap} we know that using those parameters we can define the column containing $a$ in $N(v)$ as a set $C$ and similarly
		the row containing $a$ is a definable set $R$.
		For all $n∈ℕ$ let $D_n:=\{u∈M:d(u,v)=n\}$, a definable set.
		Finally there is a definable bijection $f_n\times g_n$ for each $n∈ℕ$ such that   $D_n \mapsto \binom{C}{n} \times \binom{R}{n}$.
		Both components, $f_n$ and $g_n$, are definable.
%
		Since the domains of the $f_n$ for each $n$ are pairwise distinct we define $f=⋃_{i∈ℕ} f_n$, a mapping from $⋃_{n∈ℕ}D_n$ to the power set of $C$.
		Similarly $g=⋃_{i∈ℕ}g_n$ is a mapping from $⋃_{n∈ℕ}D_n$ to the power set of  $R$.

		Let $G$ be the pointwise stabilizer of $v,a$ and $b$ in $\operatorname{Aut}(M)$.
		We then have that $G$ fixes $R$ and $C$ as sets and induces at least
		$\operatorname{Sym}(C∖\{a,b\}) \times \operatorname{Sym}(R∖\{a,b\})$.
	Let $F⊆M$ be a finite set of parameters and $H$ be an $F$-definable subset of $C∪R$.
		Let 	$X=⋃_{c∈F} f(c)$ and
			$Y=⋃_{c∈F} g(c)$.
		Now take the pointwise stabilizer $G_{(X∪Y)}$ of $X∪Y$.
		Note that since each element of $F$ has coordinates from $X∪Y$ we have that  $G_{(X∪Y)}$ fixes $F$ pointwise and therefore fixes $H$ setwise.
		We also have that $G_{(X∪Y)}$ is transitive on $C∖(X∪\{a,b\})$.
		Thus $H∩C$ is finite or cofinite in $C$.

		Note that, since $M$ is a connected component, naming parameters in $M'∖M$ where $M'$ is an elementary extension of $M$ does not affect this finiteness, or cofiniteness.
		It follows that both the finiteness and cofinitenss are uniform across any family of definable sets.
		Thus $C$ has Morley rank $1$ and Morley degree $1$, i.\,e.\ $C$ is strongly minimal.
		Similarly $R$ is strongly minimal.

	Since $B$ is a $v$-definable set and every type over $v$ consistent with the formula defining $B$ is realised by some $e∈B$  we have $MR(B)=\max\{MR(e/v):e∈B\}$.
	Now consider a vertex  $e$ such that $d(e,v)=r$ and assume that $(f(e)∪g(e))∩\{a,b\}=∅$.
	Let $\bar{u}$ and $\bar{w}$ be the tuples from $R^r$ and $C^r$ representing  coordinates of $e$ i.e. $\bar{u}$ is a tuple representing $f(e)$ and $\bar{w}$ is a tuple representing $g(e)$.
	Then working over $vab$ the element $e$ is definable from $\bar{u}\bar{w}$,
		and $\bar{u}\bar{w}$ has finitely many conjugates over $e$.
	So by \cite[6.4.1]{MTCourse} we have that $MR(e/vab)=MR(\bar{u}\bar{w}/vab)$.
	Observe that $C$ and $R$ are strongly minimal and therefore $MR(C^t \times R^s)=t+s$
		so in particular $MR(C^r \times R^r)=2r$.
		Hence $MR(\bar{u}\bar{w}/vab)=2r=MR(e/vab)$, so $MR(B)=2r$ as required.

	\end{proof}
\end{lemma}

\begin{lemma}
	\label{SmallMorleyFinite}
	Let $ϕ$ define a small set.
	Then $ϕ$ has finite Morley rank.
	\begin{proof}
		Let $C$ be the set defined by $ϕ$.
		Since $C$ is small we have  $C⊆⋃_{i=1}^{n}B_i$ where $B_i$ are balls of finite radius.
		Let $r_i$ be the radius of $B_i$.
		Then by Lemma \ref{MorleyBall}   $B_i$ has Morley rank $2r_i$.
		Since for all $i$, $B_i$ is definable  the Morley rank of $C$ is at most $\max\{2r_i: 1≤i≤n\}$.
	\end{proof}
\end{lemma}

\begin{theorem}
	\label{T_J-MR}
	The Morley rank of a {\GenPhrase} is $ω$.
	\begin{proof}
		We note that {\GenPhrase}s contain a ball of radius $r$ for any $r∈ℕ$.
		So by Lemma \ref{MorleyBall} such a graph has Morley rank at least $ω$.
		To show that the Morley rank is $ω$ we need to show that there do not exist infinitely many pairwise disjoint definable sets each of Morley rank $ω$.
		By Lemma \ref{SmallMorleyFinite} such a collection would need to consist of large sets.
		So it is sufficient to show that there does not exist such a collection.

		Let $A$ and $B$ be large definable sets defined by $ϕ(x,(a_i)_{i∈I})$ and $ψ(x,(b_i)_{i∈J})$ respectively.
		Since $A$ and $B$ are large, by Lemma \ref{LongDistanceLemma} there are some $r_a$ and $r_b$ such that
		$ϕ(x,\bar{a})$ holds for any $x$ such that $\min\{d(x,a_i):i∈I\}≥r_a$ and
		$ψ(x,\bar{b})$ holds for any $x$ such that $\min\{d(x,b_i):i∈J\}≥r_b$.
		Let $r=\max(r_a,r_b)$. Then the set  $C:=\{x \colon \min(x,a_i)≥r ∧ \min(x,b_i)≥r\}$
		is a large set such that $C⊆A∩B$.
		As this holds for any large definable sets any two large definable sets intersect.
		In particular there does not exist a collection of infinitely many pairwise disjoint definable large sets.
	\end{proof}
\end{theorem}

\begin{corollary}
	\label{Stability}
	The theory $T_{\Johnson}$ is $ω$-stable so in particular $T_{\Johnson}$ is dependent.
\end{corollary}

We now show that this cannot be extended to monadic dependence.
Recall that
a theory is {\em monadically dependent} if any expansion of any model of $T$ by unary predicates is dependent.

\begin{theorem}
The Johnson graph $J(\aleph_0,2)$ is not monadically dependent.
\begin{proof}
We view the Johnson graph as a structure in the language $L$ of graphs (with a single binary relation $E$).
Let $ℕ$ be the underlying set of $J(\aleph_0,2)$. We define the following sets of vertices.
$$A=\{\{2n+1,0\}:n∈ ℕ\},$$
$$B=\{\{2n+2,0\}:n∈ ℕ\},$$
$$C=\{\{2n+1, 2m+2\}: m,n∈ ℕ\}.$$

Let $ψ(x,y,z)$ be the formula `$x,y,z$ are distinct and form a maximal clique of $J(\aleph_0,2)$', i.\,e.
$$\psi(x,y,z):= Exy ∧ Exz ∧ Eyz ∧ ¬∃w (Exw ∧ Eyw ∧Ezw).$$
Let $L^{A,B,C}$ be the expansion of $L$ by unary predicates interpreted by $A,B,C$.
Let $χ(x,y,z)$ be the $L^{A,B,C}$-formula $ψ(x,y,z)∧ A(x)∧ B(y)∧ C(z)$.

Thus $χ$ determines a bijection $A\times B \to C$, with $χ$ holding precisely for the triples
$(\{2n+1,0\}, \{2m+2,0\}, \{2n+1,2m+2\})$.
It follows that each element of $C$ determines via $χ$ a bipartite graph on $A ∪ B$ with a single edge,
	and each subset $C'$ of $C$ determines a bipartite graph on $A ∪ B$,
	and all bipartite graphs arise in this way.
In particular, by naming a suitable subset $C'$ of $C$ by a further unary predicate, we can define a bipartite graph witnessing the independence property (e.\,g.~the random bipartite graph).
\end{proof}
\end{theorem}

The above argument can be easily adjusted to prove the following statements.
\begin{corollary}
	\begin{itemize}
	\item
		Let $k>1$ be a natural number and $κ$ an infinite cardinal.
		Then $J(κ,k)$ is not monadically dependent.
	\item
		Let $G$ be a {\GenPhrase}. Then $G$ is not monadically dependent.
	\item
		The theory $T_{\Johnson}$ is not monadically dependent.
	\end{itemize}
\end{corollary}



\section{Open questions}
Johnson graphs are the most easily defined of several parameterised families of finite distance-transitive graphs, and it would be interesting to carry out a similar analysis of the model theory of other such families as the parameters grow.
The most obvious examples are Hamming graphs $H(d,q)$ for which the vertex set has form $S^d$ where $S$ is a set of size $q$ with two vertices adjacent if they agree in all but one coordinate.
We remark that the infinitary hypercube is shown to be $ω$-stable in Proposition 4.10 of \cite{VC-I}.
The infinitary hypercube is the `finite support' induced subgraph of $H(ω,2)$. 

There is a further class of graphs whose limit theory would be worth exploring.
Let $𝔽$ be a finite field. Then the graph $\operatorname{PGJ}(n,k,\mathbb F)$ has as vertices all of the $k$-dimensional subspaces of an $n$-dimensional vector space over $𝔽$, and two vertices are adjacent if they intersect in a $(k-1)$-dimensional space.
For a fixed $𝔽$, if the limit theory of such a class of graphs is complete then a further goal is to describe its infinite models and their stability-theoretic properties.

Other natural classes of graphs to investigate from this point of view are the `set-defined graphs' considered in \cite{Adenwalla-etal2026}. These are graphs definable just from the equality relation in a pure set. The goal would be to investigate families of set-defined graphs in {\em finite} sets, where there is a natural family structure (like `all Johnson graphs') but certain parameters are unbounded.

\bibliographystyle{Heimildastíll}
\bibliography{Heimildir}
\end{document}